\documentclass[12pt]{amsart}
\usepackage{amsmath,latexsym,amsfonts,amssymb,amsthm}
\usepackage{geometry}
\usepackage{hyperref}
\usepackage{mathrsfs}
\usepackage{graphicx,color}
\usepackage{mathtools}
\usepackage[all,cmtip]{xy}
\usepackage[alphabetic]{amsrefs}

\newtheorem{prop}{Proposition}[section]
\newtheorem{thm}[prop]{Theorem}
\newtheorem{cor}[prop]{Corollary}

\newtheorem{lem}[prop]{Lemma}

\theoremstyle{definition}

\newtheorem{defn}[prop]{Definition}
\newtheorem{expl}[prop]{Example}
\newtheorem{rem}[prop]{\it Remark}

\newtheorem*{claim*}{Claim}

\newcommand{\bC}{\mathbb{C}}
\newcommand{\bR}{\mathbb{R}}
\newcommand{\bA}{\mathbb{A}}
\newcommand{\bQ}{\mathbb{Q}}
\newcommand{\bZ}{\mathbb{Z}}
\newcommand{\bN}{\mathbb{N}}

\newcommand{\bG}{\mathbb{G}}
\newcommand{\bT}{\mathbb{T}}

\newcommand{\tX}{\widetilde{X}}

\newcommand{\tF}{\widetilde{F}}

\newcommand{\tG}{\widetilde{\Gamma}}

\newcommand{\cX}{\mathcal{X}}

\newcommand{\cD}{\mathcal{D}}
\newcommand{\cO}{\mathcal{O}}
\newcommand{\cL}{\mathcal{L}}
\newcommand{\cI}{\mathcal{I}}

\newcommand{\cF}{\mathcal{F}}
\newcommand{\cG}{\mathcal{G}}

\newcommand{\fm}{\mathfrak{m}}
\newcommand{\fX}{\mathfrak{X}}
\newcommand{\fD}{\mathfrak{D}}

\newcommand{\Supp}{\mathrm{Supp}}

\newcommand{\lct}{\mathrm{lct}}

\newcommand{\Pic}{\mathrm{Pic}}
\newcommand{\vol}{\mathrm{vol}}
\newcommand{\ord}{\mathrm{ord}}

\newcommand{\Gr}{\mathrm{Gr}}
\newcommand{\Bs}{\mathrm{Bs}}

\newcommand{\wt}{\mathrm{wt}}
\newcommand{\Val}{\mathrm{Val}}
\newcommand{\Aut}{\mathrm{Aut}}
\newcommand{\Fut}{\mathrm{Fut}}

\numberwithin{equation}{section}

\begin{document}

\title{Product theorem for K-stability}

\author{Ziquan Zhuang}
\address{Department of Mathematics, MIT, Cambridge, MA, 02139.}
\email{ziquan@mit.edu}

\date{}

\maketitle

\begin{abstract}
    We prove a product formula for delta invariant and as an application, we show that product of K-(semi, poly)stable Fano varieties is also K-(semi, poly)stable.
\end{abstract}

\section{Introduction}

K-(poly)stability of complex Fano varieties was first introduced by Tian \cite{Tian-K-stability-defn} and later reformulated in a more algebraic way by Donaldson \cite{Don-K-stability-defn}. By the generalized Yau-Tian-Donaldson (YTD) conjecture, K-polystability of (singular) Fano varieties are expected to give algebraic characterization of the existence of (singular) K\"ahler-Einstein metric. This has been known in the smooth case \cites{Tian-K-stability-defn,Berman-polystable,CDS,Tian} and the uniformly K-stable case \cite{LTW-uniform-YTD}.

From this metric point of view, it is easy to see (or at least expect) that products of K-(semi, poly)stable Fano varieties are also K-(semi, poly)stable. Results of this type actually play an important role towards the proof of the quasi-projectivity of the K-moduli \cite{CP-cm-positivity}. However, no algebraic proof is known for this intuitively simple fact.

The purpose of this note is to give such a proof. Our main result goes as follows.

\begin{thm} \label{main:product}
Let $X_i$ $(i=1,2)$ be $\bQ$-Fano varieties and let $X=X_1\times X_2$. Then $X$ is K-semistable $($resp. K-polystable, K-stable, uniformly K-stable$)$ if and only if $X_i$ $(i=1,2)$ are both K-semistable $($resp. K-polystable, K-stable, uniformly K-stable$)$.
\end{thm}

Indeed, our result works for products of log Fano pairs as well (see Corollary \ref{cor:product thm} and Proposition \ref{prop:polystable product}).

One of the main tools that goes into the proof is the $\delta$-invariant (or adjoint stability threshold) of a big line bundle (see Section \ref{sec:prelim-delta}). This invariant was introduced and studied by \cites{FO-delta,BJ-delta}, and one of their main results is that a $\bQ$-Fano variety $X$ is K-semistable (resp. uniformly K-stable) if and only if $\delta(-K_X)\ge 1$ (resp. $\delta(-K_X)>1$). This allows us to reduce most parts of Theorem \ref{main:product} to proving a product formula for $\delta$-invariant (c.f. \cite{PW-dP-delta}*{Conjecture 1.10}, \cite{CP-cm-positivity}*{Conjecture 4.9}):

\begin{thm}[=Theorem \ref{thm:delta product}] \label{main:delta product}
Let $(X_i,\Delta_i)$ be projective klt pairs and let $L_i$ be big line bundles on $X_i$ $(i=1,2)$. Let $X=X_1\times X_2$, $L=L_1\boxtimes L_2$ and $\Delta=\Delta_1\boxtimes \Delta_2$. Then
    \begin{enumerate}
        \item $\delta(X,\Delta;L)=\min\{ \delta(X_1,\Delta_1;L_1),\delta(X_2,\Delta_2;L_2)\}$.
        \item If there exists a divisor $E$ over $X$ which computes $\delta(X,\Delta;L)$, then for some $i\in\{1,2\}$, there also exists a divisor $E_i$ over $X_i$ that computes $\delta(X_i,\Delta_i;L_i)$.
    \end{enumerate}
\end{thm}

In particular, this takes care of the product of K-(semi)stable and uniformly K-stable Fano varieties. We note that the analogous product formula for Tian's alpha invariant is well known (see e.g. \cite{Hwang-alpha-product}*{Section 2}, \cite{CS-lct-Fano3fold}*{Lemma 2.29} or \cite{KP-projectivity}*{Proposition 8.11}) and indeed our proof takes inspirations from these works.

For the K-polystable case, we study K-semistable special degenerations of the product to K-semistable Fano varieties and with the help of \cite{LWX18}, we show that they always arise from special degenerations of the factors:

\begin{thm}[=Theorem \ref{thm:tc on product}] \label{main:tc on product}
Let $(X_i,\Delta_i)$ $(i=1,2)$ be K-semistable log Fano pairs and let $(X,\Delta)=(X_1\times X_2, \Delta_1\boxtimes \Delta_2)$. Let $\phi:(\cX,\cD)\rightarrow \bA^1$ be a special test configuration of $(X,\Delta)$ with K-semistable central fiber $(\cX_0,\cD_0)$, then there exists special test configurations $\phi_i:(\cX_i,\cD_i)\rightarrow \bA^1$ $(i=1,2)$ of $(X_i,\Delta_i)$ with K-semistable central fiber such that $(\cX,\cD)\cong (\cX_1\times_{\bA^1} \cX_2, \cD_1\boxtimes \cD_2)$ $($as test configurations, where $\bG_m$ acts diagonally on $\cX_1\times_{\bA^1} \cX_2)$. 
\end{thm}

Let us briefly explain the ideas of proof as well as the organization of the paper. Section \ref{sec:prelim} put together some preliminary materials on valuations, filtrations, $\delta$-invariant and K-stability. Since $\delta$-invariant is defined using log canonical threshold of basis type divisors, it is not hard to imagine that Theorem \ref{main:delta product} follows from inversion of adjunction and it suffices to show that any basis type divisors can be reorganized into one that restricts to a convex combination of basis type divisors on one of the factors. This is done in Section \ref{sec:product formula} using some auxiliary basis type filtrations constructed in Section \ref{sec:prelim-basis-type}. To address K-polystability, we analyze divisors that compute the $\delta$-invariants. We do so by choosing a maximal torus $\bT$ in the automorphism group of the Fano variety and restricting to $\bT$-invariant divisor. In this setting, equivariant K-polystability behaves somewhat like K-stability and one can give very explicit description of divisors computing $\delta$-invariants. This is made more precise in Section \ref{sec:polystable}. Once we know that product of K-polystable Fano varieties are still K-polystable, since every K-semistable Fano variety has a unique K-polystable degeneration by \cite{LWX18}, the K-semistable degenerations in Theorem \ref{main:tc on product} can be obtained by deforming the K-polystable degenerations (which is a product). But deformations of product of Fano varieties are still product of Fano varieties (see Section \ref{sec:deform product}), this gives the proof of Theorem \ref{main:tc on product}.

\subsection*{Acknowledgement}

The author would like to thank his advisor J\'anos Koll\'ar for constant support, encouragement and numerous inspiring conversations. He also wishes to thank Yuchen Liu and Chenyang Xu for helpful discussions and the anonymous referee for helpful comments. This material is based upon work supported by the National Science Foundation under Grant No. DMS-1440140 while the author was in residence at the Mathematical Sciences Research Institute in Berkeley, California, during the Spring 2019 semester. 
\section{Preliminary} \label{sec:prelim}

\subsection{Notation and conventions}

We work over the field $\bC$ of complex numbers. Unless otherwise specified, all varieties are assumed to be normal. We follow the terminologies in \cite{KM98}. A fibration is a morphism with connected fibers. A projective variety $X$ is $\bQ$-\emph{Fano} if $X$ has klt singularities and $-K_X$ is ample. A pair $(X,\Delta)$ is \emph{log Fano} if $X$ is projective, $-K_X-\Delta$ is $\bQ$-Cartier ample and $(X,\Delta)$ is klt. Let $(X,\Delta)$ be a pair and $D$ a $\bQ$-Cartier divisor on $X$, the \emph{log canonical threshold}, denoted by $\lct(X,\Delta;D)$ (or simply $\lct(X;D)$ when $\Delta=0$), of $D$ with respect to $(X,\Delta)$ is the largest number $t$ such that $(X,\Delta+tD)$ is log canonical. Let $X_i$ $(i=1,2,\cdots,m)$ be varieties over $S$, let $D_i$ be $\bQ$-divisors on $X_i$ and let $X=X_1\times_S \cdots\times_S X_m$ with projections $\pi_i:X\to X_i$, then we denote the divisor $\sum_{i=1}^m \pi_i^*D_i$ by $D_1\boxtimes\cdots\boxtimes D_m$. If $L$ is a $\bQ$-Cartier divisor on a variety $X$, we set $M(L)$ to be the set of integers $r$ such that $rL$ is Cartier and $H^0(X,rL)\neq 0$.

\subsection{Valuations}

Let $X$ be a variety. A valuation on $X$ will mean a valuation $v: K(X)^\times \to \bR$ that is trivial on the base field $\bC$. We write $\Val_X$ for the set of valuations on $X$ that also has center on $X$. A valuation $v$ is said to be divisorial if there exists a divisor $E$ over $X$ such that $v=c\cdot\ord_E$ for some $c\in\bQ_{>0}$. Let $(X,\Delta)$ be a pair. We write 
\[
A_{X,\Delta}\colon \Val_X\to \bR_{\geq 0} \cup \{ +\infty \}
\]
for the log discrepancy function with respect to  $(X,\Delta)$ as in \cites{JM-valuation,BdFFU}. We may simply write $A_X(\cdot)$ if $\Delta=0$. In particular, $A_{X,\Delta}(c\cdot\ord_E)=c\cdot A_{X,\Delta}(E)$ where $A_{X,\Delta}(E)$ is the usual log discrepancy of $E$ with respect to $(X,\Delta)$ (see e.g. \cite{Kol-mmp}*{Definition 2.4}). If $L$ is a line bundle on $X$, $v\in \Val_X$ and $s\in H^0(X,L)$, we can define $v(s)$ by trivializing $L$ at the center of $v$ and set $v(s)=v(f)$ where $f$ is the local function corresponding to $s$ under this trivialization (this is independent of choice of trivialization).

\begin{lem} \label{lem:restrict val}
Let $X\dashrightarrow X'$ be a dominant rational map of varieties and let $K'\subseteq K$ be the corresponding inclusion of their functions fields. Let $v$ be a divisorial valuation on $X$. Then its restriction to $K'$ is either trivial or a divisorial valuation on $X'$.
\end{lem}

\begin{proof}
This is well known to experts but we provide a proof for reader's convenience (c.f. \cite[Lemma 4.1]{BHJ}). Let $v'$ be the restriction of $v$ to $K'$. By the Abhyankar-Zariski inequality, we have
\[
{\rm tr. deg}(v)+{\rm rat. rk}(v)\le {\rm tr. deg}(v')+{\rm rat. rk}(v')+{\rm tr. deg}(K/K')
\]
where tr.deg (resp. rat.rk) denotes the transcendence degree (resp. rational rank) of the valuation. Since $v$ is divisorial, we have ${\rm rat. rk}(v)=1$ and ${\rm tr. deg}(v)=\dim X-1$, thus by the above inequality we obtain
\[
{\rm tr. deg}(v')+{\rm rat. rk}(v') \ge \dim X'.
\]
Since the reverse inequality always holds by Abhyankar-Zariski inequality and ${\rm rat. rk}(v')\le {\rm rat. rk}(v)=1$, we see that either ${\rm rat. rk}(v')=0$, in which case $v'$ is trivial; or ${\rm rat. rk}(v')=1$ and ${\rm tr. deg}(v')=\dim X'-1$, in which case $v'$ is a divisorial valuation by a theorem of Zariski (see e.g. \cite[Lemma 2.45]{KM98}).
\end{proof}

Let $\bT=\bG_m^r$ be a torus and let $X$ be a $\bT$-variety (i.e. a variety with a faithful action of $\bT$). Then for any $\bT$-invariant open affine subset $X_0$ of $X$ and any $f\in k[X_0]$ we have a weight decomposition
\[f=\sum_{\lambda\in M(\bT)} f_{\lambda}\]
where $M(\bT)\cong \bZ^r$ is the character group of $\bT$. In particular, let $N(\bT)$ be the lattice of one parameter subgroups of $\bT$, then for any $\xi\in N(\bT)_\bR := N(\bT)\otimes_\bZ \bR \cong \bR^r$, we can associate a $\bT$-invariant valuation
\[\wt_\xi (f) = \min_{\lambda\in M(\bT),\, f_\lambda\neq 0} \lambda\cdot \xi\]
on $X$ using the natural paring $M(\bT)\times N(\bT)\to \bZ$. It is divisorial if and only if $\xi\in N(\bT)_\bQ := N(\bT)\otimes_\bZ \bQ$. Let $K=k(X)^\bT$, then any valuation $v$ on $X$ induces a valuation $r(v)$ on $K$ by restriction. On the other hand, for any $\bT$-invariant valuation $v$ on $X$ and any $\xi\in N(\bT)_\bR$, it is not hard to check that (see e.g. \cite[Section 11]{AIPSV})
\[v_\xi (f):=\min_{\lambda\in X(\bT)} (v(f_{\lambda})+\lambda\cdot \xi)\]
defines another $\bT$-invariant valuation on $X$. This defines an action of $N(\bT)_\bR$ on the set of $\bT$-invariant valuations: $\xi\circ v\mapsto v_\xi$.

\begin{lem} \label{lem:T-valuation}
    \begin{enumerate}
        \item Any valuation $v_0$ on $K$ extends to a $\bT$-invariant valuation $v$ on $X$ such that $r(v)=v_0$.
        \item If $v$, $w$ are $\bT$-invariant valuations on $X$ such that $r(v)=r(w)$, then there exists $\xi\in N(\bT)_\bR$ such that $w=v_\xi$. In addition, $w$ is divisorial if $v$ is divisorial and $\xi\in N(\bT)_\bQ$.
    \end{enumerate}
\end{lem}

\begin{proof}
$X$ is $\bT$-equivariantly birational to $Y\times \bT$ for some variety $Y$ (on which $\bT$ acts trivially) with $K=k(Y)$, thus it suffices to prove the lemma when $X=Y\times \bT$, in which case both statements follows from an inductive use of \cite[Lemma 4.2]{BHJ}.
\end{proof}



\subsection{Filtrations}

Let $V$ be a finite dimensional vector space. A filtration $\cF$ of $V$ is given by a family of vector subspaces $\cF^\lambda V$ ($\lambda\in\bR$) such that
\begin{enumerate}
    \item $\cF^\lambda V \subseteq \cF^\mu V$ whenever $\lambda\ge \mu$;
    \item $\cF^0 V=V$ and $\cF^\lambda V=0$ for $\lambda\gg 0$;
    \item For all $\lambda\in\bR$, $\cF^\lambda V = \cF^{\lambda-\epsilon}$ for some $\epsilon>0$ depending on $\lambda$.
\end{enumerate}
It is called an $\bN$-filtration if $\cF^\lambda V=\cF^{\lceil \lambda \rceil} V$ for all $\lambda\in \bR$.

Let $L$ be an ample line bundle on a projective variety $X$ of dimension $n$. Let
\[
R:=R(X,L)=\bigoplus_{M\in\bN} R_m = \bigoplus_{M\in\bN} H^0(X,mL)
\]
be the section ring of $L$. A ($\bN$-)filtration $\cF$ of $R$ is defined as a collection of ($\bN$-)filtrations $\cF^\bullet R_m$ of $R_m$ such that $\cF^\lambda R_m\cdot \cF^\mu R_\ell \subseteq \cF^{\lambda+\mu} R_{m+\ell}$ for all $\lambda,\mu\in\bR$ and all $m,\ell\in \bN$. A filtration $\cF$ of $R$ is said to be linearly bounded if there exists some constant $C>0$ such that $\cF^{Cm}R_m=0$ for all $m\in\bN$. As a typical example, every valuation $v\in\Val_X$ induces a filtration $\cF_v$ on $R$ by setting $\cF^\lambda R_m=\{s\in R_m\,|\, v(s)\ge \lambda \}$. When $v=\ord_E$ is divisorial (where $E$ is a divisor over $X$), the induced filtration is linearly bounded; in this case we also denote the filtration by $\cF_E$.

\subsection{$\delta$-invariant} \label{sec:prelim-delta}

Let $(X,\Delta)$ be a klt pair and let $L$ be a $\bQ$-Cartier $\bQ$-divisor on $X$. Let $M(L)$ be the set of integers $m$ such that $mL$ is Cartier and $H^0(X,mL)\neq 0$. A divisor $D\sim_\bQ L$ is said to be an $m$-basis type $\bQ$-divisor of $L$ if there exists a basis $s_1,\cdots,s_{N_m}$ (where $N_m=\dim H^0(X,mL)$) of $H^0(X,mL)$ such that 
\[
D=\frac{1}{mN_m}\sum_{i=1}^{N_m} \{s_i=0\}.
\]
If $\cF$ is a filtration on $H^0(X,mL)$, an $m$-basis type $\bQ$-divisor as above is said to be compatible with $\cF$ if every subspace $\cF^\lambda H^0(X,mL)$ is spanned by some $s_i$ (c.f. \cite{AZ-K-adjunction}*{Definitions 1.5 and 2.18}). Let $v\in\Val_X$ be a valuation such that $A_{X,\Delta}(v)<\infty$. Following \cite{BJ-delta}, we define 
\[
S_m(L;v):=\sup_D v(D)
\]
where the supremum runs over all $m$-basis type $\bQ$-divisors $D$ of $L$ and set 
\[
S(L;v):=\lim_{m\to\infty} S_m(L;v).
\]
If $E$ is a divisor over $X$, we also set $\cF_E=\cF_{\ord_E}$ and $S(L;E)=S(L;\ord_E)$. Note that $\cF_E$ is an $\bN$-filtration and if $\pi:Y\to X$ is a birational morphism such that $Y$ is smooth and $E$ is a divisor on $Y$, then 
\[
S(L;E) = \frac{1}{(L^n)} \int_0^\infty \vol(\pi^*L-xE)\, {\rm d} x.
\]
We will simply write $S(v)$ or $S(E)$ if the divisor $L$ is clear from the context. It is easy to see that $S_m(v)=v(D)$ for any $m$-basis type $\bQ$-divisors $D$ that's compatible with $\cF_v$.

\begin{defn}[\cites{FO-delta,BJ-delta}] \label{defn:delta}
The $\delta$-invariant (or adjoint stability threshold) of $L$ is defined as 
\[
\delta(L):=\limsup_{m\in M(L),\, m\to \infty} \delta_m(L)
\]
where $\delta_m(L)$ is the largest $t>0$ such that $(X,\Delta+tD)$ is lc for all $m$-basis type $\bQ$-divisor $D\sim_\bQ L$. Occasionally the notation $\delta(X,\Delta;L)$ is also used to indicate which pair we are using. If $(X,\Delta)$ is a log Fano pair, we also define $\delta(X,\Delta):=\delta(-K_X-\Delta)$.
\end{defn}

\begin{thm}[\cite{BJ-delta}*{Theorems A,C and Proposition 4.3}] \label{thm:delta as inf}
Notation as above and assume that $L$ is a big line bundle on $X$. Then the above limsup is a limit and we have
\[
\delta_m(L) = \inf_E \frac{A_{X,\Delta}(E)}{S_m(E)} = \inf_{v} \frac{A_{X,\Delta}(v)}{S_m(v)},\quad\delta(L) = \inf_E \frac{A_{X,\Delta}(E)}{S(E)} = \inf_{v} \frac{A_{X,\Delta}(v)}{S(v)}
\]
where in both equalities the first infimum runs through all divisors $E$ over $X$ and the second through all $v\in \Val_X$ with $A_{X,\Delta}(v)< +\infty$.
\end{thm}

In view of this theorem, we say that a divisor $E$ over $X$ computes $\delta(L)$ if $\delta(L)=\frac{A_{X,\Delta}(E)}{S(E)}$. 

\subsection{K-stability}

We refer to \cites{Tian-K-stability-defn,Don-K-stability-defn} for the original definition of K-stability. Here we define this notion using valuations and $\delta$-invariant. The equivalence of this definition with the original one is shown by the work of \cites{Fujita-valuative-criterion,FO-delta,Li-equivariant-minimize,BJ-delta,LWX18,BX-separatedness}.

\begin{defn}
Let $(X,\Delta)$ be a log Fano pair. A \emph{special test configuration} $(\cX,\cD)/\bA^1$ of $(X,\Delta)$ consists of the following data:
    \begin{enumerate}
        \item a normal variety $\cX$, a flat projective morphism $\pi\colon\cX \to \bA^1$, together with an effective $\bQ$-divisor $\cD$ on $\cX$ that does not contain any fiber of $\pi$ in its support such that $-(K_\cX+\cD)$ is $\pi$-ample;
        \item a $\bG_m$-action on $(\cX,\cD)$ such that $\pi$ is $\bG_m$-equivariant with respect to the standard action of $\bG_m$ on $\bA^1$ via multiplication;
        \item $(\cX,\cD)\times_{\bA^1} (\bA^1\setminus\{0\})$
        is $\bG_m$-equivariantly isomorphic to $(X,\cD)\times(\bA^1\setminus\{0\})$;
        \item $(\cX,\cX_0+\cD)$ is plt where $\cX_0=\pi^{-1}(0)$.
    \end{enumerate}
A special test configuration is called a \emph{product} test configuration if $(\cX,\cD)\cong(X,\Delta)\times\bA^1$. 
\end{defn}

We say that $(X,\Delta)$ \emph{specially degenerates to} $(X_0,\Delta_0)$ if there exists a special test configuration of $(X,\Delta)$ with central fiber $(X_0,\Delta_0)$ (by adjunction, it is a log Fano pair). By \cite[Lemma 3.1]{LWX18}, a special test configuration $(\cX,\cD)\to \bA^1$ has K-semistable central fiber $(\cX_0,\cD_0)$ if and only if $\Fut(\cX,\cD)=0$ where $\Fut(\cX,\cD)$ is the generalized Futaki invariant (sometimes called Donaldson-Futaki invariant) of the test configuration. 

\begin{defn}
Let $(X,\Delta)$ be a log Fano pair. It is 
    \begin{enumerate}
        \item K-semistable if $\delta(X,\Delta)\ge 1$;
        \item K-stable if $A_{X,\Delta}(E)>S(E)$ for all divisors $E$ over $X$;
        \item uniformly K-stable if $\delta(X,\Delta)>1$;
        \item K-polystable if it is K-semistable and any K-semistable special degeneration $(X_0,\Delta_0)$ of $(X,\Delta)$ comes from a product test configuration.
    \end{enumerate}
\end{defn}

The following statement is a reformulation of \cite[Theorem 1.4]{LWX18}.

\begin{thm}[\cite{LWX18}] \label{thm:T-polystable}
Let $(X,\Delta)$ be a log Fano pair and let $\bT$ be a maximal torus in $\Aut(X,\Delta)$. Then $(X,\Delta)$ is K-polystable if and only if it is K-semistable and $A_{X,\Delta}(v)>S(v)$ for all $\bT$-invariant divisorial valuations $v$ unless $v=\wt_\xi$ for some $\xi\in N(\bT)_\bQ$.
\end{thm}

\begin{proof}
By definition and Theorem \ref{thm:delta as inf} we have $A_{X,\Delta}(v)\ge S(v)$ for all divisorial valuations $v$ whenever $(X,\Delta)$ is K-semistable. By \cite[Theorem 1.4]{LWX18}, $(X,\Delta)$ is K-polystable if and only if it's $\bT$-equivariantly K-polystable, i.e. in the definition of K-polystability, it suffices to consider $\bT$-equivariant special test configurations. By \cite[Theorem 4.1]{BX-separatedness}, ($\bT$-equivariant) special degenerations of $(X,\Delta)$ to K-semistable log Fano pairs correspond to ($\bT$-invariant) divisorial valuations $v$ for which $A_{X,\Delta}(v)=S(v)$. Since $\bT$ is a maximal torus in $\Aut(X,\Delta)$, $\bT$-equivariant product test configurations all come from one parameter subgroups of $\bT$, thus correspond to valuations of the form $v=\wt_\xi$ for some $\xi\in N(\bT)_\bQ$.
\end{proof}

\begin{lem} \label{lem:tc over A^2}
Let $(X,\Delta)$ be a K-semistable log Fano pair and let $\phi_i:(\cX_i,\cD_i)\to \bA^1$ $(i=1,2)$ be two special test configurations of $(X,\Delta)$ with K-semistable central fibers. Then there exists a $\bG_m^2$-equivariant projective morphism $\psi:(\fX,\fD)\to \bA^2$ such that
    \begin{enumerate}
        \item $-(K_\fX+\fD)$ is $\bQ$-Cartier and for all $t\in \bA^2$, the fibers $(\fX_t,\fD_t)$ are K-semistable log Fano pairs;
        \item $(\fX,\fD)\times_{\bA^2} \bG_m^2 \cong (X,\Delta)\times \bG_m^2$;
        \item $(\fX,\fD)\times_{\bA^2} (\bA^1\times \{1\})\cong (\cX_1,\cD_1)$ over $\bA^1$ and similarly $(\fX,\fD)\times_{\bA^2} (\{1\} \times \bA^1)\cong (\cX_2,\cD_2)$.
    \end{enumerate}
\end{lem}

\begin{proof}
This is a more precise version of \cite[Theorem 3.2]{LWX18} and essentially follows from the proof of \cite[Theorem 3.2]{LWX18}.
\end{proof}

We will also use the following result from \cite{Li-singular-YTD}.

\begin{lem} \label{lem:twist valuation compute delta}
Let $\bT$ be a torus and let $(X,\Delta)$ be a K-semistable log Fano pair with a $\bT$-action. Let $v\in\Val_X$ be a $\bT$-invariant valuation such that $A_{X,\Delta}(v)=S(v)$. Then we have $A_{X,\Delta}(v_\xi)=S(v_\xi)$ for all $\xi\in N(\bT)_\bR$ such that $v\neq \wt_{-\xi}$.
\end{lem}

\begin{proof}
This is a direct consequence of \cite{Li-singular-YTD}*{Proposition 3.12} as $\Fut_{(X,\Delta)}(\xi)=0$ by the K-semistability of $(X,\Delta)$. 
\end{proof}

\subsection{Basis type filtrations}  \label{sec:prelim-basis-type}

Let $V$ be a finite dimensional vector space. 

\begin{defn}
A basis type filtration $\cF$ of $V$ is an $\bN$-filtration
\[
V=\cF^0 V\supseteq \cF^1 V \supseteq \cdots \supseteq \cF^N V = 0
\]
such that $\dim \Gr_\cF^i V = 1$ for all $0\le i\le N-1$ (in particular, $N=\dim V$).
\end{defn}

In the actual application, we will always take $V=H^0(X,L)$ for some line bundle $L$ on a projective variety $X$. The following construction of basis type filtrations are particularly important for us.

\begin{expl} \label{expl:basic construction}
Let $V=H^0(X,L)$ as above and let $N=\dim V$. We construct a basis type filtration $\cF$ of $V$ as follows. Let $\cF^0 V=V$. Suppose that $\cF^i V$ has been constructed, we view it as a linear series and write
\[
|\cF^i V| = |M_i| + F_i
\]
where $F_i$ is the fixed part and $M_i$ is the movable part. Choose a smooth point $x_{i+1}\in X$ that's not a base point of $|M_i|$, then evaluating at $x_{i+1}$ gives a surjective map $M_i\rightarrow M_i\otimes k(x_{i+1})$ and we denote its kernel by $M_i\otimes \fm_{x_{i+1}}$ (it consists of those elements of $M_i$ that vanishes at $x_{i+1}$). We then define $\cF^{i+1} V$ by the formula $|\cF^{i+1} V| = |M_i\otimes \fm_{x_{i+1}}| + F_i$. It is clear that $\cF^{i+1} V$ has codimension $1$ in $\cF^i V$. The construction of the filtration then proceeds inductively. We call the resulting filtration the basis type filtration associated to the prescribed base points $x_1,\cdots,x_N$.
\end{expl}

We will mainly use two special cases of the above construction.

\begin{expl}
The construction clearly works if $x_1,\cdots,x_N$ are distinct general points on $X$, in which case the associated basis type filtration $\cF$ of $V$ is said to be of type (I). 
\end{expl}

\begin{expl}
As a variant, let $\pi:Y\rightarrow X$ be a proper birational morphism and let $E$ be a divisor on $Y$. Recall that we have a filtration $\cF_E$ on $V=H^0(Y,\pi^*L)=H^0(X,L)$ given by $\cF_E^i V=H^0(Y,\pi^*L-iE)$. In the construction in Example \ref{expl:basic construction}, since the $M_i$'s are movable, we may choose $x_1,\cdots,x_N$ to be distinct general points on $E$ and it is not hard to see that the associated basis type filtration $\cF$ is a refinement of $\cF_E$. We call it a basis type filtration of type (II) associated to the divisor $E$.
\end{expl}

We note the following elementary property of basis type filtration.

\begin{lem} \label{lem:filtration property}
Let $V$ be a vector space of dimension $N$. Let $\cF, \cG$ be two $\bN$-filtrations on $V$ where $\cF$ is of basis type. Let $i\in\bN$ and let \[A_i=\{j\in\bN \,|\,\dim \Gr_\cF^j \Gr_\cG^i V = 1 \}.\]
Then $|A_i|=\dim \Gr_\cG^i V$ and $\cup_{i=0}^{\infty} A_i$ gives a partition of $\{0,1,\cdots,N-1\}$.
\end{lem}

\begin{proof}
Since $\cF$ is of basis type, the induced filtration on $\Gr_\cG^i V$ satisfies $\dim \Gr_\cF^j \Gr_\cG^i V \le 1$ for all $j\in\bN$, thus $|A_i|=\dim \Gr_\cG^i V$. It is not hard to check that
\[
\Gr_\cF^j \Gr_\cG^i V \cong (\cF^j V\cap \cG^i V) / (\cF^{j+1}V\cap \cG^i V + \cF^j V\cap \cG^{i+1}V) \cong \Gr_\cG^i \Gr_\cF^j V.
\]
Since $\dim \Gr_\cF^j V = 1$ for all $0\le j\le N-1$, for any such $j$ there exists a unique $i\in\bN$ such that $\dim \Gr_\cG^i \Gr_\cF^j V = 1$. By the above equality, this implies that the $A_i$'s give a partition of $\{0,1,\cdots,N-1\}$.
\end{proof}

\subsection{Deformations of product of Fano varieties} \label{sec:deform product}

The results in this section are probably well known but we cannot find a suitable reference (but see \cite{Li-deform-product}).

\begin{lem} \label{lem:nef in nbd}
Let $(\cX,\cD)$ be a klt pair and $f:\cX\to B$ a flat projective fibration to a smooth variety, and let $0\in B$ be a point such that the fiber $\cX_0$ is a normal variety not contained in the support of $D$. Let $\cL$ be a $\bQ$-Cartier $\bQ$-divisor on $\cX$ such that $\cL|_{\cX_0}$ is nef and $(a\cL - K_{\cX/B}-\cD)|_{\cX_0}$ is nef and big for some $a\ge 0$. Then $\cL|_{\cX_b}$ is nef for all $b$ in a Zariski neighbourhood of $0\in B$.
\end{lem}

\begin{proof}
The proof is essentially the same as \cite[Lemma 3.9]{dFH-deform-Fano}. Replacing $\cL$ by $m\cL$ for some sufficiently divisible $m$, we may assume that $\cL$ is Cartier. By \cite[Lemma 3.9]{dFH-deform-Fano} (applied to the divisor $\cL'=2a\cL - K_{\cX/B}-\cD$), we may also assume that $(2a\cL - K_{\cX/B}-\cD)|_{\cX_b}$ is nef and big for all $b$ after shrinking $B$. By \cite[Proposition 1.4.14]{Laz-positivity-1}, $\cL|_{\cX_b}$ is nef if $b\in B$ is very general. This means there is a set $W\subseteq B$ that is the complement of countably many subvarieties such that $\cL|_{\cX_b}$ is nef when $b\in W$. By \cite[Theorem 1.1]{Kol-eff-bpf}, there exists a positive integer $m$ (depending only on $a$ and $n=\dim \cX_0$) such that $|m\cL|_{\cX_b}|$ is base point free for all $b\in W$. Further shrinking $B$ and $W$ if necessary, we may assume that the restriction map $H^0(\cX,m\cL)\to H^0(\cX_b,m\cL|_{\cX_b})$ is surjective for all $b\in W$. But since $|m\cL|_{\cX_b}|$ is base point free, we conclude that $\Bs(m\cL)\cap \cX_b=\emptyset$ for $b\in W$. As $f$ is proper, it follows that there exists an open set $U\subseteq B$ containing $W$ so that $\Bs(m\cL)\cap \cX_b=\emptyset$ whenever $b\in U$ and hence $|m\cL|_{\cX_b}|$ is base point free when $b\in U$. In particular, $\cL|_{\cX_b}$ is nef over $U$. We are done if $0\in U$; otherwise, replace $B$ by a resolution of the components of $B\backslash U$ and use Noetherian induction.
\end{proof}

\begin{lem} \label{lem:local topology}
Let $f:\cX\to B$ be a projective morphism, let $b\in B$ and let $X_b={\rm red}\,f^{-1}(b)$ be the reduced fiber over $b$. Then there exists an analytic neighbourhood $U\subseteq B$ of $b$ such that the natural map $H^i(\cX_U,\bZ)\to H^i(X_b,\bZ)$ $($where $\cX_U=\cX\times_B U)$ is an isomorphism for all $i\in \bN$.
\end{lem}

\begin{proof}
By choosing a triangulation of $\cX$ and $B$ such that $X_b$ is a sub-complex and $f$ is a map between CW complexes (see e.g. \cite{Loj-triangulation,Hir-triangulation}), we see that there exists an analytic neighbourhood $U\subseteq B$ of $b$ such that $\cX_U$ deformation retracts to $X_b$. Therefore, the maps $H^i(\cX_U,\bZ)\to H^i(X_b,\bZ)$ are isomorphisms.
\end{proof}

\begin{lem} \label{lem:deform Fano product}
Let $X_i$ $(i=1,2)$ be normal projective varieties and let $X=X_1\times X_2$. Let $\Delta$ be an effective $\bQ$-divisor on $X$ such that $(X,\Delta)$ is log Fano. Let $(\cX,\cD)$ be a pair and let $\phi:\cX\to B$ be a flat projective fibration onto a smooth variety $B$ such that the support of $\cD$ does not contain any fiber of $\phi$. Let $0\in B$ and assume that $(\cX_0,\cD_0)=(\phi^{-1}(0),\cD|_{\cX_0})$ is isomorphic to $(X,\Delta)$. Then
    \begin{enumerate}
        \item There exists an open set $($in the analytic topology$)$ $0\in U\subseteq B$ and two projective morphisms $\cX_i\to U$ $(i=1,2)$ with central fibers $X_i$ such that $\cX\times_B U \cong \cX_1\times_U \cX_2$ over $U$. Moreover, both $\cX_i$ $(i=1,2)$ are uniquely determined by $\phi$ and $U$.
        \item If $B=\bA^r$ and $(\cX,\cD)$ admits a $\bG_m^r$-action such that $\phi:\cX\to \bA^r$ is $\bG_m^r$-equivariant, then one can take $U=\bA^r$ in $(1)$ and moreover, the factors $\cX_i$ also admit $\bG_m^r$-actions making the isomorphism $\cX \cong \cX_1\times_{\bA^r} \cX_2$ equivariant.
    \end{enumerate}
\end{lem}

\begin{proof}
We may assume that $B$ is affine and (using inversion of adjunction) that $(\cX_b,\cD_b)$ is log Fano for all $t\in B$ possibly after shrinking $B$ in (1) (since the family is equivariant in (2), no shrinking is necessary). By Kawamata-Viehweg vanishing we have $H^i(\cX_b,\cO_{\cX_b})=0$ for all $b\in B$ and all $i>0$, hence as $B$ is affine, $H^i(\cX,\cO_{\cX})=H^i(B,\phi_*\cO_{\cX})=0$ for all $i>0$ as well. By the long exact sequence associated to the exponential sequence $0\to \bZ\to \cO_{\cX}\to \cO_{\cX}^*\to 1$, we see that $\Pic(\cX)\cong H^2(\cX,\bZ)$ and $\Pic(\cX_0)\cong H^2(\cX_0,\bZ)$. By Lemma \ref{lem:local topology}, after further shrinking $B$, the natural map $H^2(\cX,\bZ)\to H^2(\cX_0,\bZ)$ (and hence $\Pic(\cX)\to \Pic(\cX_0)$ as well) is an isomorphism. In case (2), no shrinking is necessary since the diagonal $\bG_m$-action (corresponding to the inclusion $\bG_m\to \bG_m^r$, $t\mapsto (t,t,\cdots,t)$) already induces a deformation retract of $\cX$ onto $X_0$, hence also isomorphisms in integral cohomology and Picard groups.

In particular, let $M_i$ ($i=1,2$) be an ample line bundle on $X_i$ and let $\pi_i:X\rightarrow X_i$ be the natural projection, then $L_i=\pi_i^*M_i$ extends to a line bundle $\cL_i$ on $\cX$. Since the extension is unique, $\cL_i$ is $\bG_m^r$-invariant in case (2). As $(\cX_b,\cD_b)$ is log Fano, by Lemma \ref{lem:nef in nbd} and Shokurov's base-point-free theorem, $\cL_i$ is $\phi$-nef and $\phi$-semiample after possibly shrinking $B$ in (1). Let $\psi_i=\psi_{|m\cL_i|}:\cX \rightarrow \cX_i$ ($i=1,2$) be the fibration (over $B$) induced by the linear system $|m\cL_i|$ for sufficiently large and divisible $m$ and let $\psi=\psi_1\times_B \psi_2: \cX\rightarrow \cX_1\times_B \cX_2$. Note that in case (2), these maps are $\bG_m^r$-equivariant. By Kawamata-Viehweg vanishing we have $H^1(\cX,m\cL_i\otimes \cI_{\cX_0})=H^1(B,\fm_0\cdot \pi_*(m\cL_i))=0$ as before, thus $H^0(\cX,m\cL_i)$ surjects onto $H^0(\cX_0,m\cL_i|_{\cX_0})=H^0(X,mL_i)$. It follows that $\psi_i|_{\cX_0}$ is given by the projection $X\rightarrow X_i$ and $\psi|_{\cX_0}$ is the isomorphism $X\stackrel{\sim}{\rightarrow} X_1\times X_2$, thus $\psi|_{\cX_b}$ is also an isomorphism for all $b\in B$ (possibly after shrinking $B$ in case (1)). 

It remains to prove the uniqueness of the factors. Suppose that we have a decomposition $\cX\times_B U\cong \cX_1\times_U \cX_2$ with central fibers $\cX_{i,0}\cong X_i$ $(i=1,2)$. As $H^j(\cX,\cO_{\cX})=0$ for all $j>0$, by K\"unneth formula we have $H^j(\cX_i,\cO_{\cX_i})=0$ $(i=1,2)$ as well. Thus as in the above proof we have isomorphisms $\Pic(X_i)\cong \Pic(\cX_i)$. Therefore the ample line bundle $M_i$ chosen above uniquely lifts to $\cX_i$ and its pullback to $\cX$ coincides with $\cL_i$ (again by the uniqueness of the extension of $L_i$ to $\cX$). It follows that $\cX_i$ is uniquely determined as the image of $\cX$ under the map $\psi_i=\psi_{|m\cL_i|}$.
\end{proof}
\section{Product formula for delta invariant} \label{sec:product formula}

\begin{thm} \label{thm:delta product}
Let $(X_i,\Delta_i)$ be projective klt pairs and let $L_i$ be big line bundles on $X_i$ $(i=1,2)$. Let $X=X_1\times X_2$, $L=L_1\boxtimes L_2$ and $\Delta=\Delta_1\boxtimes \Delta_2$. Then
    \begin{enumerate}
        \item $\delta(X,\Delta;L)=\min\{ \delta(X_1,\Delta_1;L_1),\delta(X_2,\Delta_2;L_2)\}$.
        \item If there exists a divisor $E$ over $X$ which computes $\delta(X,\Delta;L)$, then for some $i\in\{1,2\}$, there also exists a divisor $E_i$ over $X_i$ that computes $\delta(X_i,\Delta_i;L_i)$.
    \end{enumerate}
\end{thm}

\begin{proof}[Proof of Theorem \ref{thm:delta product}]
For simplicity we assume that $\Delta_1=\Delta_2=0$; the proof of the general case is almost identical. It is easy to see that 
\begin{equation} \label{eq:delta<=min}
    \delta(L) \le \min\{ \delta(L_1),\delta(L_2)\},
\end{equation}
so for (1) we only need to prove the reverse inequality. Let
\begin{equation} \label{eq:c<delta}
    0<c<\min\{ \delta(L_1),\delta(L_2)\}
\end{equation}
and let $E$ be a divisor over $X$ (living on some smooth birational model $\pi:\tX \to X$). By Theorem \ref{thm:delta as inf} we need to show that
\[
S_m(E)\le c^{-1} A_X(E)
\]
when $m\gg 0$. Let $R_m=H^0(X,mL)$, let $\cF$ be any basis type filtration of $R_m$ given by refining the filtration $\cF_E$ of $R_m$ and let $D$ be any $m$-basis type $\bQ$-divisor of $L$ that's compatible with $\cF$, then we have $S_m(E)=\ord_E(D)$ (since $D$ is also compatible with $\cF_E$) and therefore it suffices to show that
\begin{equation} \label{claim:klt along E}
    (X,cD) \text{ is klt along } E \text{ when } m\gg 0.
\end{equation}
Note that this claim does not depend on the choice of $\cF$ and $D$. Let $R_{m,i}=H^0(X_i,mL_i)$, let $N_{m,i}=\dim R_{m,i}$ ($i=1,2$) and let $N_m=\dim R_m$. For ease of notation we also let $N=N_{m,2}$. By K\"unneth formula, we have $R_m=R_{m,1}\otimes R_{m,2}$, thus $N_m = N_{m,1}N_{m,2}$.

Assume first that the center of $E$ on $X$ dominates $X_2$. Let $\cG$ be a basis type filtration of $R_{m,2}$ of type (I) associated to some prescribed base points $x_1,\cdots,x_N$. After tensoring with $R_{m,1}$, it induces an $\bN$-filtration (which we also denote by $\cG$) on $R_m$. By construction, we have canonical isomorphisms
\begin{equation} \label{eq:restriction}
    \Gr_\cG^i R_m \cong R_{m,1}\otimes \Gr_\cG^i R_{m,2} \cong H^0(X_1,mL_1)\otimes k(x_{i+1})
\end{equation}
for $0\le i\le N-1$. Now $\cF$ induces a filtration on the graded pieces $\Gr_\cG^i R_m$ and since $\cF$ is of basis type, we have $\dim \Gr_\cF^j \Gr_\cG^i R_m\le 1$ for all $i,j$. For a fixed $0\le i\le N-1$, let $A_i=\{j\,|\,\dim \Gr_\cF^j \Gr_\cG^i R_m = 1\}$. By Lemma \ref{lem:filtration property}, we have $|A_i|=N_{m,1}$ for all $i$ and $\cup_{i=0}^{N-1} A_i$ is a partition of $\{0,1,\cdots,N_m-1\}$. Let $x\in X_2$ be a general smooth point and let $F=X_1\times x\subseteq X$. For $0\le i\le N-1$ and $j\in A_i$, let $f_j$ be a general member of $\cF^j R_m \cap \cG^i R_m$. We claim that for a fixed $i$, 
\begin{equation} \label{claim:form basis}
    f_j|_F\;(j\in A_i) \text{ form a basis of } H^0(F,mL_1).
\end{equation}
Indeed, by our construction, $f_j|_{X_1\times x_{i+1}}$ form a basis of $H^0(X_1,mL_1)$ via the surjection $\cF^j R_m \cap \cG^i R_m \twoheadrightarrow \Gr_\cF^j \Gr_\cG^i R_m$ and the isomorphism \eqref{eq:restriction}, hence the same holds over a general point $x\in X_2$, proving \eqref{claim:form basis}.

It follows from \eqref{eq:c<delta} and \eqref{claim:form basis} that the pair (where $D_j=(f_j=0)$)
\[(F,\frac{c}{mN_{m,1}}\sum_{j\in A_i} D_j|_F)\]
is klt when $m\gg 0$. By inversion of adjunction, this implies that \[(X,\frac{c}{mN_{m,1}}\sum_{j\in A_i} D_j)\]
is klt in a neighbourhood of $F$. As being klt is preserved under convex combination, we see that $(X,c\Gamma_m)$ is also klt near $F$ where $\Gamma_m:=\frac{1}{mN_m}\sum_{j=0}^{N_m-1} D_j = \frac{1}{mN_m} \sum_{i=0}^{N-1} \sum_{j\in A_i} D_j$. In particular, it is klt along the divisor $E$; from the construction it is not hard to see that $\Gamma_m$ is an $m$-basis type $\bQ$-divisor of $L$ that is compatible with $\cF$, this proves \eqref{claim:klt along E} when $E$ dominates $X_2$.

Suppose that $E$ computes $\delta(L)$, i.e. 
\begin{equation} \label{eq:compute delta}
    \delta(L)=\frac{A_X(E)}{S(E)}.
\end{equation}
Let $\delta=\delta(L)$ and let $\tF$ be the strict transform of $F$ on $\tX$. Then $\tF$ is a log resolution of $F$ and $E|_{\tF}$ is a smooth divisor on $\tF$. Let $E_1$ be an irreducible component of $E|_{\tF}$. Since $x\in X_2$ is general, for a fixed $m$ we have $A_F(E_1)=A_X(E)$ and $\ord_{E_1}(\Gamma_m|_F)=\ord_E(\Gamma_m)$. Since $\Gamma_m$ is compatible with $\cF_E$, letting $m\rightarrow \infty$ we have $\ord_E(\Gamma_m)=S_m(E)\to S(E) = \delta^{-1} A_X(E)$ by \eqref{eq:compute delta}. Hence if $x\in X_2$ is very general, we have $\ord_{E_1}(\Gamma_m|_F)\to \delta^{-1} A_F(E_1)$. But by \eqref{claim:form basis}, $\Gamma_m|_F$ is a convex combination of $m$-basis type divisors, thus we have
\[S(E_1) \ge \lim_{m\to \infty} \ord_{E_1}(\Gamma_m|_F) = \delta^{-1} A_F(E_1)\]
and therefore by identifying $F$ with $X_1$, we get a chain of inequalities
\[\delta \ge \frac{A_F(E_1)}{S(E_1)} \ge \delta(L_1) \ge \delta, \]
where the last inequality comes from \eqref{eq:delta<=min}. It follows that equalities hold throughout and hence $E_1$ computes $\delta(X_1,L_1)$.

Next assume that the center of $E$ on $X$ does not dominate $X_2$. By Lemma \ref{lem:restrict val}, $\ord_E$ induces a divisorial valuation $v$ on $X_2$ via the projection $X\rightarrow X_2$. Let $\phi:Y\rightarrow X_2$ be a birational morphism such that $Y$ is smooth and the center of $v$ on $Y$ is a divisor $G$. Let $\cG$ be a basis type filtration on $R_{m,2}=H^0(Y,m\pi^*L_2)$ of type (II) associated to some general points $x_1,\cdots,x_N$ on $G$. As in the previous case, we get an induced filtration $\cG$ on $R_m$ and a canonical isomorphism \eqref{eq:restriction} for $0\le i\le N-1$. We also have the induced filtration $\cF$ on the graded pieces $\Gr_\cG^i R_m$ and we define the sets $A_i$ ($0\le i\le N-1$) and choose $f_j\in \cF^j R_m \cap \cG^i R_m$ ($j\in A_i$) as before. Let $D_j=(f_j=0)\subseteq X$ and let $W=X_1\times Y$, with the induced birational map $W\rightarrow X$ still denoted by $\phi$. We may write $\phi^*D_j=a_j \pi_2^*G+B_j$ for some $a_j\ge 0$ where $\pi_2$ is the second projection $W\rightarrow Y$ and $\pi^*_2 G\not\subseteq \Supp(B_j)$. By the construction of $\cG$, we have $a_j=\ord_G (\cG^i R_{m,2})$ if $j\in A_i$. Now let $x$ be a general point of $G$ and let $F=X_1\times x\subseteq W$. Note that $B_j\sim m\phi^*L-a_j\pi_2^* G$, hence $B_j|_F\sim mL_1$. We claim that for a fixed $0\le i\le N-1$,
\begin{equation} \label{claim:form basis-II}
    B_j|_F\;(j\in A_i) \text{ form a basis of } |mL_1|.
\end{equation}
Indeed, by the construction of $\cG$ and the isomorphism \eqref{eq:restriction}, $B_j|_{X_1\times x_{i+1}}$ form a basis of $|mL_1|$, hence the same is true for a general point $x$ and \eqref{claim:form basis-II} follows.

Let $\Gamma_m=\frac{1}{mN_m}\sum_{j=0}^{N_m-1} D_j$ as before. We may write
\begin{equation} \label{eq:pullback}
    K_W + q_m(c) \pi_2^*G+\tG_m=\phi^*(K_X+c\Gamma_m)
\end{equation}
for some $q_m(c) \in \bQ$ and some divisor $\tG_m$ (it is not necessarily effective but is effective near $F$) not containing $\pi^*_2 G$ in its support. In fact, from the previous discussions we have
\begin{eqnarray*}
q_m(c) & = & \frac{c}{mN_m}\sum_{j=0}^{N_m-1}a_j - A_{X_2}(G) +1\;\; = \;\; \frac{c}{mN_m} \sum_{i=0}^{N-1} \sum_{j\in A_i} a_j - A_{X_2}(G) +1 \\
 & = & \frac{c}{mN_m} \sum_{i=0}^{N-1} N_{m,1}\cdot \ord_G(\cG^i R_{m,2}) - A_{X_2}(G) +1 \\
 & \to & c\cdot S(\ord_G) - A_{X_2}(G) +1 \;\; < \;\; 1 \;\; (m\to \infty)
\end{eqnarray*}
where the convergence comes from the fact that the basis type filtration $\cG$ is a refinement of $\cF_G$ and the last inequality holds because $c<\delta(L_2)$. Taking $m\gg 0$, we may then assume that $q_m(c)<1$. Recall that the center of $E$ dominates $G$ and $F$ is the fiber over a general point of $G$, thus to prove \eqref{claim:klt along E}, it suffices to show that $(X,c\Gamma_m)$ is klt near $F$, which follows if we know that $(W,\pi_2^*G+\tG_m)$ is plt near $F$. But it is not hard to see that $\tG_m|_F=\frac{c}{mN_m}\sum_{j=0}^{N_m-1} B_j|_F$, hence $(F,\tG_m|_F)$ is klt (when $m\gg 0$) by \eqref{eq:c<delta} and \eqref{claim:form basis-II} as in the previous case. \eqref{claim:klt along E} now follows by inversion of adjunction. In particular, we have proven the first statement of the theorem.

Suppose that $E$ computes $\delta(L)$. We claim that $G$ computes $\delta(L_2)$. Suppose that this is not the case, then $S(G)\cdot \delta(L_2) <A_{X_2}(G)$, hence by the above computation, there exists some constant $\epsilon>0$ such that $q_m(c)<1-\epsilon$ for all $c<\delta(L_2)$ and all corresponding $m\gg 0$. Since $(W,\pi_2^*G+\tG_m)$ is plt near $F$ when $m\gg 0$, we have
\begin{eqnarray*}
A_{X,c\Gamma_m}(E) & = & A_{W, q_m(c) \pi_2^*G+\tG_m}(E) \\
 & = & A_{W,\pi_2^*G+\tG_m}(E) + (1-q_m(c))\cdot \ord_E(G) \\
 & > & \epsilon \cdot \ord_E(G).
\end{eqnarray*}
Letting $m\rightarrow \infty$ and then $c\to \delta=\delta(L)$, we obtain
\[A_X(E)-\delta \cdot S(E) \ge \epsilon \cdot \ord_E(G) > 0,\]
a contradiction to \eqref{eq:compute delta}. This finishes the proof.
\end{proof}

\begin{rem} \label{rem:explicit div on factors}
It follows from the above proof that if $E$ computes $\delta(L)$, then either the center of $E$ dominates $X_2$ and its restriction to a very general fiber $X_1\times x$ gives a divisor $E_1$ over $X_1$ that computes $\delta(L_1)$, or the center of $E$ doesn't dominate $X_2$ and induces a divisorial valuation (through the second projection) on $X_2$ that computes $\delta(L_2)$. In the former case, we can actually say a bit more:
\end{rem}

\begin{cor} \label{cor:induce div dominant case}
Notation as in Theorem \ref{thm:delta product}. Let $E$ be a divisor over $X$ that computes $\delta(L)$ whose center dominates $X_2$, then for a general $x\in X_2$, the restriction of $E$ to $X_1\times x$ induces a prime divisor that computes $\delta(L_1)$.
\end{cor}

\begin{proof}
As before we assume that $\Delta_1=\Delta_2=0$. Let $\phi:\tX\rightarrow X$ be a log resolution on which $E$ lives as an actual divisor as in the above proof. Let $x\in X_2$ be a general point, then the strict transform $\tF_x$ of $F_x:=X_1\times x$ is smooth, $E|_{\tF_x}$ is a smooth divisor and $A_{F_x}(\ord_{E_x})=A_X(\ord_E)$ for any component $E_x$ of $E|_{\tF_x}$. By the proof of Theorem \ref{thm:delta product}, we have
\[\delta(L_1) = \frac{A_{F_y}(\ord_{E_y})}{S_{F_y}(\ord_{E_y})}\]
for very general points $y\in X_2$. 
But by the upper semi-continuity of volume function, we have 
\begin{eqnarray*}
    S_{F_x}(\ord_{E_x}) & = & \frac{1}{\vol(L_1)}\int_0^\infty \vol_{\tF_x}(\phi^*L_1-tE_x){\rm d}t \\
     & \ge & \frac{1}{\vol(L_1)}\int_0^\infty \vol_{\tF_y}(\phi^*L_1-tE_y){\rm d}t \\
     & = & S_{F_y}(\ord_{E_y}).
\end{eqnarray*}
Hence we also have 
\[\delta(L_1) \ge \frac{A_{F_x}(\ord_{E_x})}{S_{F_x}(\ord_{E_x})}.\]
Since the reverse inequality clearly holds, it's indeed an equality and thus $E_x$ computes $\delta(L_1)$.
\end{proof}

\begin{cor} \label{cor:product thm}
Let $(X_i,\Delta_i)$ $(i=1,2)$ be log Fano pairs and let $(X,\Delta)=(X_1\times X_2, \Delta_1\boxtimes \Delta_2)$. Then $(X,\Delta)$ is K-semistable $($resp. K-stable, uniformly K-stable$)$ if and only if $(X_i,\Delta_i)$ $(i=1,2)$ are both K-semistable $($resp. K-stable, uniformly K-stable$)$.
\end{cor}

\begin{proof}
By definition, $(X,\Delta)$ is K-semistable (resp. uniformly K-stable) if and only if $\delta(X,\Delta)\ge 1$ (resp. $>1$), thus the statement in these cases follows from the above product formula of $\delta$-invariant. On the other hand, $X$ is K-stable if and only if $\delta(X,\Delta)>1$ or $\delta(X,\Delta)=1$ and it is not computed by any divisorial valuations, hence the result again follows from Theorem \ref{thm:delta product}.
\end{proof}
\section{K-polystable case} \label{sec:polystable}

In this section, we prove the K-polystable part of Theorem \ref{main:product}.

\begin{prop} \label{prop:polystable product}
Let $(X_i,\Delta_i)$ $(i=1,2)$ be log Fano pairs and let $(X,\Delta)=(X_1\times X_2, \Delta_1\boxtimes \Delta_2)$. Then $(X,\Delta)$ is K-polystable if and only if $(X_i,\Delta_i)$ $(i=1,2)$ are both K-polystable.
\end{prop}
 
\begin{proof}
The ``only if'' part is obvious so we only prove the ``if'' part. Assume that $(X_i,\Delta_i)$ are both K-polystable. Let $\bT_i$ ($i=1,2$) be a maximal torus of $\Aut(X_i,\Delta_i)$, then $\bT=\bT_1\times\bT_2$ is a maximal torus of $X$. By Theorem \ref{thm:T-polystable}, we need to show that if $E$ is a $\bT$-invariant divisor over $X$ with $A_{X,\Delta}(E)=S(E)$, then $\ord_E=\wt_\xi$ for some $\xi\in N(\bT)$. As in the proof of Theorem \ref{thm:delta product}, we separate into two cases.
 
First suppose that the center of $E$ dominates $X_2$. By Corollary \ref{cor:induce div dominant case}, over a general $x\in X_2$, $E$ induces a $\bT_1$-invariant divisor $E_x$ over $X_1\times x$ that computes $\delta(X_1,\Delta_1)$; i.e., $A_{X_1,\Delta_1}(E_x)=S(E_x)$. By Theorem \ref{thm:T-polystable}, this implies $\ord_{E_x}=\wt_{\xi_x}$ for some $\xi_x\in N(\bT_1)$. But as the $E_x$ varies in a continuous family, $\xi_x$ is constant and hence we have $\ord_E=\wt_\xi$ for some $\xi\in N(\bT_1)\subseteq N(\bT)$.
 
Next suppose that the center of $E$ does not dominate $X_2$. Then by Remark \ref{rem:explicit div on factors}, $E$ induces a divisor $G$ over $X_2$ such that $A_{X_2,\Delta_2}(G)=S(G)$. By Theorem \ref{thm:T-polystable}, this implies that $\ord_G=\wt_{\xi_2}$ for some $\xi_2\in N(\bT_2)$. By Lemma \ref{lem:twist valuation compute delta}, either $\ord_E=\wt_{\xi_2}$ and there is nothing to prove or $v=(\ord_E)_{-\xi_2}$ computes $\delta(X,\Delta)=1$ (note that $(X,\Delta)$ is K-semistable by Corollary \ref{cor:product thm}). By Lemma \ref{lem:T-valuation}, $v$ is divisorial and we have $v=b\cdot \ord_{E'}$ for some divisor $E'$ over $X$. Notice that the center of $E'$ dominates $X_2$, otherwise if $G'$ is the divisor on $X_2$ induced by $E'$ then the divisorial valuation induced by $\ord_E$ on $X_2$ should be $(b\cdot\ord_{G'})_{\xi_2}$ rather than $\wt_{\xi_2}$. But then from the discussion of the previous case, we have $\ord_{E'}=\wt_{\xi_1}$ for some $\xi_1\in N(\bT_1)\subseteq N(\bT)$. It follows that $\ord_E=(b\cdot\wt_{\xi_1})_{\xi_2}=\wt_{b\xi_1+\xi_2}$.
\end{proof}

\begin{thm} \label{thm:tc on product}
Let $(X_i,\Delta_i)$ $(i=1,2)$ be K-semistable log Fano pairs and let $(X,\Delta)=(X_1\times X_2, \Delta_1\boxtimes \Delta_2)$. Let $\phi:(\cX,\cD)\rightarrow \bA^1$ be a special test configuration of $(X,\Delta)$ with K-semistable central fiber $(\cX_0,\cD_0)$, then there exists special test configurations $\phi_i:(\cX_i,\cD_i)\rightarrow \bA^1$ $(i=1,2)$ of $(X_i,\Delta_i)$ with K-semistable central fiber such that $(\cX,\cD)\cong (\cX_1\times_{\bA^1} \cX_2, \cD_1\boxtimes \cD_2)$ $($as test configurations, where $\bG_m$ acts diagonally on $\cX_1\times_{\bA^1} \cX_2)$. 
\end{thm}

\begin{proof}
By \cite[Theorem 1.3]{LWX18}, $(X_i,\Delta_i)$ has a (unique) K-polystable special degeneration $(Y_i,\Gamma_i)$. By Proposition \ref{prop:polystable product}, $(Y,\Gamma)=(Y_1\times Y_2, \Gamma_1\boxtimes \Gamma_2)$ is K-polystable, hence by Corollary \ref{cor:product thm} and \cite[Theorem 1.3]{LWX18}, it is the unique K-polystable special degeneration of the K-semistable log Fano pair $(X,\Delta)$. By Lemma \ref{lem:tc over A^2}, this degeneration can be put into a $\bG_m^2$-equivariant family $\psi:(\fX,\fD)\to \bA^2$ with K-semistable log Fano fibers such that $(\fX,\fD)\times_{\bA^2} \bG_m^2 \cong (X,\Delta)\times \bG_m^2$, $(\fX,\fD)\times_{\bA^2} (\bA^1\times \{1\})\cong (\cX,\cD)$ (over $\bA^1$) and $(\fX,\fD)\times_{\bA^2} \{(1,0)\} \cong (Y,\Gamma)$. Since $(Y,\Gamma)$ is K-polystable and specially degenerates to the K-semistable pair $(\fX_0,\fD_0)$ over $0\in\bA^2$, we get $(\fX_0,\fD_0)\cong (Y,\Gamma)$ and $(\fX,\fD)\times_{\bA^2} (\bA^1\times \{0\})$ is a product test configuration. As $(Y,\Gamma)$ is log Fano and $Y=Y_1\times Y_2$ is a product, by Lemma \ref{lem:deform Fano product}, there exists $\bG_m^2$-equivariant morphisms $\fX_i\to \bA^2$ ($i=1,2$) with central fibers $Y_i$ such that $\fX \cong \fX_1\times_{\bA^2} \fX_2$ equivariantly over $\bA^2$. Restricting to $\bA^1\times \{0\}$ we see that $\fX_i\times_{\bA^2} \{(1,0)\}\cong Y_i$ by the uniqueness part of Lemma \ref{lem:deform Fano product}. Similarly, as $(\fX,\fD)\times_{\bA^2} (\{1\}\times\bA^1)$ is the fiber product of two special degenerations $X_i\rightsquigarrow Y_i$ by construction, we have $\fX_i\times_{\bA^2} \{(1,1)\}\cong X_i$ by Lemma \ref{lem:deform Fano product}. Denote by $p_i:\fX\to \fX_i$ ($i=1,2$) the projections onto the factors under this isomorphism, let $\pi_i:X\to X_i$ be the natural projections and let $\fD'_i\subseteq \fX$ be the closure of $\pi_i^*\Delta_i\times \bG_m^2$. Let $\fD_i=p_i(\fD'_i)$. By construction, we have $\fD_{i,0}\cong \Gamma_i$, thus by upper semi-continuity of fiber dimension we obtain $\dim \fD_{i,t}\le \dim \fD_{i,0} = \dim \fX_i -1$ for all $t\in \bA^2$, thus $\fD_i$ is a divisor in $\fX_i$ and $\fD'_i$ is the pullback of $\fD_i$. In particular, we have $(\fX,\fD)\cong (\fX_1\times_{\bA^2} \fX_2, \fD_1\boxtimes \fD_2)$ over $\bA^2$. Restricting to $\bA^1\times \{1\}$, we obtain the statement in the theorem.
\end{proof}

\bibliography{ref}

\end{document}